\documentclass{amsart}
\usepackage{amssymb}
\usepackage[bookmarks=true]{hyperref}

\newcommand{\NN}{\mathbb N}
\newcommand{\RR}{\mathbb R}
\newcommand{\diverg}[2]{D\left(#1 \Vert #2\right)}
\newcommand{\PP}{\mathbb P}

\newtheorem{thm}{Theorem}
\newtheorem{lem}[thm]{Lemma}
\newtheorem{prop}[thm]{Proposition}
\newtheorem{rmk}[thm]{Remark}

\title[An Algorithm for Finding Positive Solutions to Polynomial Equations]{An
Algorithm for Finding Positive \\ Solutions to Polynomial Equations}
\author{Dustin Cartwright}
\address{Dept.\ of Mathematics, University of California, Berkeley, CA 94720,
USA}
\email{dustin@math.berkeley.edu}

\begin{document}
\begin{abstract}
We present a numerical algorithm for finding real non-negative
solutions to polynomial equations. Our methods are based on the expectation
maximization and iterative proportional fitting algorithms, which are used in
statistics to find maximum likelihood parameters for certain classes of
statistical models. Since our algorithm works by iteratively improving an
approximate solution, we find approximate solutions in the cases when there are
no exact solutions, such as overconstrained systems.
\end{abstract}
\maketitle

We present an iterative numerical method for finding non-negative solutions and
approximate solutions to systems of polynomial equations. We require two
assumptions about our system of equations. First, for each equation, all the
coefficients other than the constant term must be non-negative. Second, there is
a technical assumption on the exponents, described at the beginning of
Section~\ref{sec:alg}, which, for example, is satisfied if all non-constant
terms have the same total degree. In Section~\ref{sec:universality}, there is a
discussion of the range of possible systems which can arise under these
hypotheses.

Because of the assumption on signs, we can write our system of equations as,
\begin{equation} \label{eqn:main}
\sum_{\alpha \in S} a_{i\alpha}
x^\alpha = b_{i} \quad\mbox{for } i = 1, \ldots, \ell,
\end{equation}
where the coefficients
$a_{i\alpha}$ are non-negative and the $b_i$ are positive, and
$S \subset \RR_{\geq 0}^n$ is a finite set of possibly non-integer
multi-indices.
Our algorithm works by iteratively decreasing the generalized Kullback-Leibler
divergence of the left-hand side and right-hand side of~(\ref{eqn:main}).
The generalized Kullback-Leibler divergence of two positive vectors $a$ and~$b$
is defined to be
\begin{equation} \label{eqn:gen-kl}
D(a \Vert b) := \sum_{i} \left( b_i \log\left(\frac{b_i}{a_i}\right) -
b_i + a_i\right).
\end{equation}
The standard Kullback-Leibler consists only of the first term and is defined
only for probability distributions, i.e.\ the sum of each vector is~$1$. The
last two terms are necessary so that the generalized divergence
has, for arbitrary positive vectors $a$ and~$b$, the property of being
non-negative and zero exactly when $a$ and~$b$ are equal
(Proposition~\ref{prop:nonneg}).

Our algorithm converges to local minima of the Kullback-Leibler divergence,
including exact solutions to the system~(\ref{eqn:main}).  In
order to find multiple local minima, we can repeat the algorithm for randomly
chosen starting points. For finding approximate solutions, this may be
sufficient.
However, there are no guarantees of completeness for the exact solutions
obtained in this way.
Nonetheless, we hope that in certain situations, our algorithm will find
applications both for finding exact and approximate solutions.

Lee and Seung applied the EM algorithm to the problem of non-negative matrix
factorization~\cite{lee-seung-alg}. They introduced the generalized
Kullback-Leibler divergence in~(\ref{eqn:gen-kl}) and used it to find
approximate non-negative matrix factorizations. Since the product of
two matrices can be expressed by polynomials in the entries of the matrices,
matrix-factorization is a special case of the equations in~(\ref{eqn:main}).

For finding exact solutions to arbitrary systems of polynomials, there are
a variety of approaches which find all complex
or all real solutions. Homotopy continuation methods
find all complex roots of a system of equations~\cite{verschelde}. 
Even to find only positive roots, these two methods finds all complex or all
real solutions, respectively.
Lasserre, Laurent, and Rostalski have applied semidefinite programming to find
all real solutions to a system of equations and a slight modificiation of their
algorithm will find all positive real solutions~\cite{llr2008,llr2009}.
Nonetheless, neither of these methods has any notion of approximate
solutions.

For directly finding only positive real solutions, Bates and
Sottile have proposed an algorithm based on fewnomials bounds on the number of
real solutions~\cite{bates-sottile}. However, their method is only effective
when the number of monomials (the set $S$ in our notation) is only slightly more
than the number of variables.  Our method only makes weak assumptions on the set
of monomials, but stronger assumptions on the coefficients.

Our inspiration comes from tools for maximum likelihood estimation in
statistics. Parameters which maximize the likelihood are exactly the
parameters such that the model probabilities are closest to the empirical
probabilities, in the sense of mimimizing Kullback-Leibler divergence.
Expectation-Maximization~\cite[Sec.\ 1.3]{ascb} and Iterative Proporitional
Fitting~\cite{darroch-ratcliff} are well-known iterative methods for maximum
likelihood estimation. We re-interpret these algorithms as methods for
approximating solutions to polynomial equations, in which case their
applicability can be somewhat generalized.

The impetus behind the work in this paper was the need to find approximate
positive solutions to systems of bilinear equations in~\cite{cartwright}. In
this case the variables represented expression levels, which only made sense as
positive parameters. Moreover, in order to accomodate noise in the data, there
were more equations than variables, so it was necessary to find approximate
solutions. Thus, the algorithm described in this paper was the most appropriate
tool. Here, we generalize beyond bilinear equations and present proofs.

An implementation of our algorithm in the C programming language
is freely available at \url{http://math.berkeley.edu/~dustin/pos/}.

In Section~\ref{sec:alg}, we describe the algorithm and the connection to
maximum likelihood estimation. In Section~\ref{sec:convergence}, we prove that
the necessary convergence for our algorithm. Finally, in
Section~\ref{sec:universality}, we show that even with our restrictions on the
form of the equations, there can be exponentially positive real solutions.

\section{Algorithm}\label{sec:alg}

We make the assumption that we have an $s \times n$ non-negative
matrix $g$, with no column identically zero, and positive real numbers~$d_j$ for
$1 \leq j \leq s$ such that for each $\alpha \in S$ and each $j \leq s$,
$\sum_{i=1}^n g_{ji} \alpha_i$ is either $0$ or $d_j$.  For example, if all the
monomials $x^\alpha$ have the same total degree $d_1$, we can take $s=1$ and
$g_{1i} = 1$ for all~$i$. The other case of particular interest is multilinear
systems of equations, in which each 
$\alpha_i$ is at most one. In this case the variables can be partitioned into
sets such that the equations are linear in each set of variables, so we can take
$d_j =1$ for all~$j$. Note that because $d_j$ is in the denominator
in~(\ref{eqn:alg-update}), convergence is fastest when the $d_j$ are small, such
as in the multilinear case. We also note that, for an
arbitrary set of exponents~$S$, there may not exist such a~$g$.

The algorithm begins with a randomly chosen starting vector and iteratively
improves it through two nested iterations:
\begin{enumerate}
\let\olditem=\item\def\item{\olditem[$\bullet$]}
\item Initialize $x$ with $n$ randomly chosen positive real numbers.
\item Loop until the vector $x$ stabilizes:
\begin{enumerate}
\item For all $\alpha \in S$, compute
\begin{equation}\label{eqn:alg-w}
w_{\alpha} := \sum_{i} b_i \frac{a_{i\alpha} x^\alpha}{\sum_{\beta}
a_{i\beta}x^\beta}.
\end{equation}
\item Loop until the vector $x$ stablizes:
\begin{enumerate}
\item Loop for $j$ from $1$ to $s$:
\begin{enumerate}
\item Simultaneously update all entries of $x$:
\begin{equation}\label{eqn:alg-update}
x_i \leftarrow x_i \left(\frac{\sum_\alpha \alpha_i g_{ji} w_{\alpha}}
{\sum_{\alpha} \alpha_i g_{ji} a_{\alpha} x^\alpha}\right)^{g_{ji}
/d_j} \quad
\mbox{where } a_\alpha = \sum_{i} a_{i\alpha}.\hspace{-1.5in} 
\end{equation}
\end{enumerate}
\end{enumerate}
\end{enumerate}
\end{enumerate}
Because there is no subtraction, it is clear that the entries of~$x$ remain
positive throughout this algorithm.

Our method is inspired by interpreting the equations in (\ref{eqn:main}) as a
maxmimum likelihood problem for a statistical model and applying the well-known
methods of Expectation-Maximization (EM) and Iterative Proportional Fitting
(IPF). Here, we assume that all the monomials $x^\alpha$ have the same total
degree. Our statistical model is that a hidden process generates an integer $i \leq \ell$ and an exponent
vector $\alpha$ with joint probability~$a_{i\alpha} x^\alpha$. The
vector
$x$ contains $n$ positive parameters for the model, restricted such that the
total probability $\sum_{i,\alpha} a_{i\alpha} x^\alpha$ is $1$. The empirical
data consists of repeated observations of the integer~$i$, but not the
exponent~$\alpha$, and $b_i$ is the proportion of observations of~$i$. In this
situation, the vector $x$ which minimizes the divergence of~(\ref{eqn:main}) is
the maximum likelihood parameters for the empirical distribution~$b_i$. The
inner loop of the algorithm consists of using IPF to solve the log-linear hidden
model and the outer loop consists of using EM to estimate the distribution on
the hidden states.

\section{Proof of convergence}\label{sec:convergence}

In this section we prove our main theorem:

\begin{thm}\label{thm:main}
The Kullback-Leibler divergence
\begin{equation}\label{eqn:diverg-main}
\sum_{i=1}^\ell \diverg{\sum_{a\in S} a_{i\alpha}x^\alpha}{b_i}.
\end{equation}
is weakly decreasing during the algorithm in Section~\ref{sec:alg}.
Moreover, assuming that the set $S$ contains a multiple of each unit
vector~$e_i$, i.e.\ some power of each~$x_i$ appears in the system of equations,
then the vector $x$ coverges to a critical point of the
function~(\ref{eqn:diverg-main}) or the boundary of the positive orthant.
\end{thm}

\begin{rmk}
The condition on~$S$ is necessary to ensure that the vector $x$ remains bounded
during the algorithm.
\end{rmk}

We begin by establishing several basic properties of the generalized
Kullback-Leibler divergence in Proposition~\ref{prop:nonneg} and
Lemmas~\ref{lem:scaling} and~\ref{lem:normalizing}.
Note that $\diverg{a}{b} =
\sum_{i=1}^m \diverg{a_i}{b_i}$ and thus, we will check these basic
properties by reducing to the case where $a$ and $b$ are scalars.

The proof of Theorem~\ref{thm:main} itself is divided into two parts,
corresponding to the two nested iterative loops. The first step is
to prove that
the updates~(\ref{eqn:alg-update}) in the inner loop converge a local minimum of
the divergence $\diverg{a_{\alpha} x^\alpha}{w_\alpha}$.  The second step is to
show that this implies that the outer loop strictly decreases the divergence
function~(\ref{eqn:diverg-main}) exact at a critical point.

\begin{prop} \label{prop:nonneg}
For $a$ and~$b$ vectors of positive real numbers, the divergence $\diverg{a}{b}$
is non-negative with $\diverg{a}{b} = 0$ if and only if $a = b$.
\end{prop}

\begin{proof}
It suffices to prove this when $a$ and $b$ are scalars. We let $t
= b/a$, and,
\begin{equation*}
\diverg{a}{b} = b \log\left(\frac{b}{a}\right) - b + a
 = a (t\log t - t + 1) = a \int_{1}^t \log s \, ds.
\end{equation*}
It is clear that $\int_1^t \log s\, ds$ is non-negative and equal to zero if and
only if $t = 1$.
\end{proof}

\begin{lem} \label{lem:scaling}
Suppose that $a$ and~$b$ are vectors of $m$ positive real numbers.
Let $t$ be any positive real number, and then
\begin{equation*}
\diverg{ta}{b} = \diverg{a}{b} + (t-1) \sum_{i=1}^m a_i - \sum_{i=1}^m b_i \log t
\end{equation*}
\end{lem}

\begin{proof}
Again, we assume that $a$ and $b$ are scalars and then it becomes a
straightforward computation.
\end{proof}

\begin{lem} \label{lem:normalizing}
If $a$ and $b$ are vectors of $m$ positive real numbers, then we 
can relate their divergence to the divergence of their sums by
\begin{equation*}
\diverg{a}{b} = \diverg{\textstyle\sum_{i=1}^m a_i}{\textstyle\sum_{i=1}^m b_i}
+ \diverg{\frac{\sum_{i=1}^m  b_i}{\sum_{i=1}^m a_i} a}{b}.
\end{equation*}
\end{lem}

\begin{proof}
We let $A= \sum_{i=1}^m a_i$ and $B= \sum_{i=1}^m b_i$, and apply
Lemma~\ref{lem:scaling} to the last term:
\begin{align*}
\diverg{\frac{B}{A}a}{b} &= \diverg{a}{b} + \left(\frac{B}{A} - 1\right)A - B
\log \frac{B}{A} \\
&= \diverg{a}{b} - B \log\frac{B}{A} + B - A = \diverg{a}{b} - \diverg{A}{B}.
\end{align*}
After rearranging, we get the desired expression.
\end{proof}

\begin{lem} \label{lem:ipf}
The update rule~(\ref{eqn:alg-update}) weakly decreases the divergence.
If
\begin{equation*}
x_i' = x_i \left(\frac{\sum_\alpha \alpha_i g_{ji} w_{\alpha}}
{\sum_{\alpha} \alpha_i g_{ji} a_{\alpha} x^\alpha}\right)^{g_{ji}
/d_j},
\end{equation*}
then
\begin{equation} \label{eqn:ipf-ineq}
\sum_\alpha \diverg{a_\alpha (x')^\alpha}{w_\alpha} \leq
\sum_\alpha \diverg{a_\alpha x^\alpha}{w_\alpha} -
\frac{1}{d}\sum_{i = 1}^n \diverg{\sum_\alpha \alpha_i g_{ji} a_\alpha
x^\alpha}
{\sum_\alpha \alpha_i g_{ji} w_\alpha}.
\end{equation}
\end{lem}

\begin{proof}
First, since the statement only depends on the $j$th row of the matrix $g$, we
can assume that $g$ is a row vector and we drop $j$ from future subscripts.
Second, we can assume that $d=1$ by replacing $g_i$ with $g_i/d$.

Third, we reduce to the case when $g_{i} = 1$ for all $i$. We define a new set
of exponents~$\tilde
\alpha$ and coefficients~$\tilde a_{\tilde \alpha}$ by $\tilde\alpha_i = g_i \alpha_i$ and 
$\tilde a_\alpha = a_\alpha \prod x_i$, where the product is taken over all
indices $i$ such that $g_i = 0$.
We take $\tilde x$ to be a vector indexed by those $i$ such that $g_i \neq
0$. Then, under the change of coordinates $\tilde
x_i = x_i^{1/g_i}$, we have $a_\alpha x^\alpha = \tilde a_\alpha \tilde
x^{\tilde \alpha}$ and
the update rule in~(\ref{eqn:alg-update}) is the same for the new system with
coefficients $\tilde a_{\tilde \alpha}$ and exponents $\tilde \alpha$.
Furthermore, if all entries of $\tilde\alpha$ are zero, then $\tilde
x^{\tilde\alpha}=1$ for all vectors~$x$ and so we can drop $\tilde\alpha$ from
our exponent set.
Therefore, for the rest of the proof, we drop the tildes, and assume that
$\sum_i \alpha_i = 1$ for all $\alpha \in S$ and $g_i =
1$ for all $i$, in which case $g$ drops out of the equations.

To prove the desired inequality, we substitute the updated assignment $x'$ into the
definition of Kullback-Leibler divergence:
\begin{align}
\diverg{a_\alpha (x')^\alpha}{w_\alpha} &=
w_\alpha \log\left(\frac{w_\alpha}{a_\alpha x^\alpha \prod_{i=1}^n
\left(\frac{\sum_{\beta} w_\beta}{\sum_\beta \beta_i a_\beta
x^\beta}\right)^{\alpha_i}}\right) - w_\alpha
+ a_\alpha (x')^\alpha \notag \\
&= w_\alpha \log\frac{w_\alpha}{a_\alpha x^\alpha}
- \sum_{i=1}^n \alpha_i w_\alpha \log\frac{\sum_{\beta} \beta_i
w_\beta}{\sum_\beta \beta_i a_\beta x^\beta}
- w_\alpha
+ a_\alpha (x')^\alpha \notag \\
&= \diverg{a_\alpha x^\alpha}{w_\alpha}
- \sum_{i=1}^n \alpha_i w_\alpha \log\frac{\sum_{\beta} \beta_i
w_\beta}{\sum_\beta \beta_i a_\beta x^\beta}
- a_\alpha x^\alpha
+ a_\alpha (x')^\alpha. \label{eqn:ipf-kl}
\end{align}

On the other hand, let $C$ denote the last term of~(\ref{eqn:ipf-ineq}), which
we can expand as,
\begin{align}
 C &= \sum_{i = 1}^n \diverg{\sum_\alpha \alpha_i a_\alpha
x^\alpha}
{\sum_\alpha \alpha_i w_\alpha} \notag \\
&= 
\sum_{i=1}^n \left(\left(\sum_{\alpha} \alpha_i w_\alpha \right)
\log\frac{\sum_{\alpha} \alpha_i w_\alpha}{\sum_{\alpha}\alpha_i a_\alpha
x^\alpha}
- \sum_\alpha \alpha_i w_\alpha
+ \sum_\alpha \alpha_i a_\alpha x^\alpha \right) \notag \\
&= \sum_{i=1}^n \sum_{\alpha} \left(\alpha_i w_\alpha
\log \frac{\sum_\beta \beta_i w_\beta}{\sum_\beta \beta_i a_\beta x^\beta}
- \alpha_i w_\alpha
+ \alpha_i a_\alpha x^\alpha \right) \notag \\
&= \sum_{\alpha} \left(\sum_{i=1}^n \alpha_i w_\alpha
\log\frac{\sum_\beta \beta_i w_\beta}{\sum_\beta \beta_i a_\beta x^\beta}\right)
- w_\alpha
+ a_\alpha x^\alpha, \label{eqn:expansion-c}
\end{align}
where the last step follows from the assumption that
that $\sum_i \alpha_i = 1$ for all $\alpha \in S$.
We take the sum of~(\ref{eqn:ipf-kl}) over all $\alpha \in S$
and add it to~(\ref{eqn:expansion-c}) to get,
\begin{equation}\label{eqn:ipf-kl-final}
\sum_\alpha \diverg{a_\alpha(x')^\alpha}{w_\alpha}
+ C
= \sum_\alpha \diverg{a_\alpha x^\alpha}{w_\alpha}
- \sum_\alpha b_\alpha
+ \sum_\alpha a_\alpha (x')^\alpha.
\end{equation}

Finally, we expand the last term of~(\ref{eqn:expansion-c}) using the definition
of $x'$ and apply the arithmetic-geometric mean inequality,
\begin{align*}
\sum_\alpha a_\alpha (x')^\alpha &=
\sum_\alpha a_\alpha x^\alpha \prod_{i=1}^n\left(
\frac{\sum_\beta \beta_i b_\beta}{\sum_\beta \beta_i a_\beta
x^\beta}\right)^{\alpha_i} \\
&\leq \sum_{\alpha} a_\alpha x^\alpha \sum_{i=1}^n
\alpha_i \frac{\sum_\beta \beta_i
b_\beta}{\sum_\beta \beta_i a_\beta x^\beta} \\
&= \sum_{i=1}^n \left(\sum_{\alpha} \alpha_i a_\alpha x^\alpha \right)
\frac{\sum_\beta \beta_i b_\beta}{\sum_\beta \beta_i a_\beta x^\beta} \\
&= \sum_{i=1}^n \sum_\beta \beta_i b_\beta = \sum_\beta b_\beta.
\end{align*}
Together with~(\ref{eqn:ipf-kl-final}), this gives the desired inequality.
\end{proof}

\begin{prop}\label{prop:ipf-fixed}
A positive vector $x$ is a fixed point of the update rule~(\ref{eqn:alg-update})
for all $1 \leq j \leq s$ if and only if $x$ is a critical point of the
divergence function $\sum_{\alpha} \diverg{a_\alpha x^\alpha}{w_\alpha}$.
\end{prop}

\begin{proof}
For the update rule to be constant means that the numerator and denominator in
(\ref{eqn:alg-update}) are equal, i.e.
\begin{equation}\label{eqn:suff-equal}
\sum_\alpha \alpha_i g_{ji} a_\alpha x^\alpha
= \sum_\alpha \alpha_i g_{ji} w_\alpha
\quad\mbox{for all $i$ and $j$.}
\end{equation}
By our assumption on $g$, for each $i$, some
$g_{ji}$ is non-zero, so~(\ref{eqn:suff-equal}) is equivalent to
\begin{equation}\label{eqn:suff-equal2}
\sum_\alpha \alpha_i a_\alpha x^\alpha
= \sum_\alpha \alpha_i w_\alpha \quad\mbox{for all } i.
\end{equation}
On the other hand, we compute the partial derivative
\begin{equation*}
\frac{\partial}{\partial x_i} \sum_\alpha \diverg{a_\alpha x^\alpha}{w_\alpha}
= \sum_{\alpha }
-w_\alpha \frac{\alpha_i}{x_i}
+ \alpha_i a_\alpha \frac{x^\alpha}{x_i}.
\end{equation*}
Since each $x_i$ is assumed to be non-zero, it is clear that all partial
derivatives being zero is equivalent to~(\ref{eqn:suff-equal2}).
\end{proof}

\begin{lem}\label{lem:em}
If we define $w_\alpha$ as in (\ref{eqn:alg-w}), then
\begin{multline*}
\sum_{i=1}^n \diverg{\sum_\alpha a_{i\alpha} (x')^\alpha}{b_i}
- \sum_{i=1}^n \diverg{\sum_\alpha a_{i\alpha} x^\alpha}{b_i} \\
\leq \sum_{\alpha} \diverg{a_{\alpha}(x')^\alpha}{w_{\alpha}}
- \sum_{\alpha} \diverg{a_{\alpha}x^\alpha}{w_{\alpha}}.
\end{multline*}
Moreover, a positive vector~$x$ is a fixed point if and only if $x$ is a
critical point for the divergence function.
\end{lem}

\begin{proof}
We consider
\begin{equation}\label{eqn:diverg-i-alpha}
\sum_{i, \alpha}\diverg{a_{i\alpha}(x')^\alpha}{w_{i\alpha}}
\qquad\mbox{where }
w_{i\alpha} = \frac{b_i a_{i\alpha} x^\alpha}{\sum_\beta
a_{i\beta} x^\beta},
\end{equation}
and apply Lemma~\ref{lem:normalizing} in two different ways.
First, by applying Lemma~\ref{lem:normalizing} to each group
of~(\ref{eqn:diverg-i-alpha}) with fixed
$\alpha$, we get
\begin{equation*}
\sum_{i, \alpha}\diverg{a_{i\alpha}(x')^\alpha}{w_{i\alpha}}
=\sum_{\alpha} \diverg{a_\alpha (x')^\alpha}{w_{\alpha}}
+ \sum_{i,\alpha} \diverg{\frac{\sum_{j} w_{j\alpha}}{\sum_j a_\alpha
(x')^\alpha} a_{i\alpha} (x')^\alpha}{w_{i\alpha}} .
\end{equation*}
In the last term, the monomials $(x')^\alpha$ cancel and so it is a constant
independent of $x'$ which we denote $E$.
On the other hand, we can apply Lemma~\ref{lem:normalizing} to each group
in~(\ref{eqn:diverg-i-alpha}) with fixed~$i$,
\begin{equation*}
\sum_{i, \alpha} \diverg{a_{i\alpha}(x')^\alpha}{w_{i\alpha}}
= \sum_i \diverg{\sum_{\alpha}a_{i\alpha}(x')^\alpha}{b_i}
+ \sum_{i,\alpha} \diverg{\frac{b_i a_{i\alpha} (x')^\alpha}
{\sum_\beta a_{i\beta} (x')^\beta}}{w_{i\alpha}}.
\end{equation*}
We can combine these equations to get
\begin{equation} \label{eqn:em-main}
\sum_{i} \diverg{\sum_{\alpha}a_{i\alpha}(x')^\alpha}{b_i}
= \sum_{\alpha}\diverg{a_{\alpha}(x')^\alpha}{w_{i\alpha}}
+ E
- \sum_{i,\alpha} \diverg{\frac{b_i a_{i\alpha} (x')^\alpha}
{\sum_{\beta} a_{i\beta} (x')^\beta}}{w_{i\alpha}}.
\end{equation}
By Proposition~\ref{prop:nonneg}, the last term of~\ref{eqn:em-main} is
non-negative, and by the definition of~$w_{i\alpha}$, it is zero for $x' = x$.
Therefore, any value of~$x'$ which decreases the
first term compared to~$x$ will also decrease the left hand side by at least as
much, which i sthe desired inequality.

In order to prove the statement about the derivative, we consider the derivative
of~(\ref{eqn:em-main}) at $x'=x$. Because the last term is mimimized at $x'=x$,
its derivative is zero, so
\begin{equation*}
\left.\frac{\partial}{\partial x_j'}\right\vert_{x'=x}
\sum_i \diverg{\sum_{\alpha} a_{i\alpha} (x')^\alpha}{b_i}
= \left.\frac{\partial}{\partial x_j'}\right\vert_{x'=x}
\sum_{i, \alpha}\diverg{a_{i\alpha}(x')^\alpha}{w_{i\alpha}}.
\end{equation*}
By Proposition~\ref{prop:ipf-fixed}, a positive vector $x$ is a fixed point of
the inner loop if and only if these partial derivatives on the right are zero
for all indices~$j$, which is the definition of a critical point.
\end{proof}

\begin{proof}[Proof of Theorem~\ref{thm:main}]
The Kullback-Leibler divergence $\sum_{\alpha}
\diverg{a_\alpha x^\alpha}{w_\alpha}$ decreases at each step of the inner loop
by Lemma~\ref{lem:ipf}.
Thus, by Lemma~\ref{lem:em}, the divergence
\begin{equation}\label{eqn:main-thm-diverg}
\sum_{i=1}^n \diverg{\sum_{\alpha} a_{i\alpha} x^\alpha}{b_i}
\end{equation}
decreases at least as much.
However, the divergence~(\ref{eqn:main-thm-diverg}) is non-negative
according to Proposition~\ref{prop:nonneg}. Therefore, the magnitude of
the decreases in divergence must approach zero over the course of
the algorithm. By Lemma~\ref{lem:ipf}, this means that the quantity $C$ in
that theorem must approach zero. By Proposition~\ref{prop:nonneg}, this means
that the quantities in that divergence approach each other. However, up to a
power, these are the numerator and denominator of the factor in the update
rule~(\ref{eqn:alg-update}), so the difference between consecutive vectors $x$
approaches zero.

Thus, we just need to show that $x$ remains bounded. However,
since some power of each variable $x_i$ occurs in some equation, as
$x_i$ gets large, the divergence for that equation also gets arbitrarily large.
Therefore, each $x_i$ must remain bounded, so the vector~$x$ must have a limit
as the algorithm is iterated. If this limit is in the interior of the positive
orthant, then it must be a fixed point. By Lemma~\ref{lem:em}
and Proposition~\ref{prop:ipf-fixed}, this fixed point must be a critical point of the
divergence~(\ref{eqn:diverg-main}).
\end{proof}

\section{Universality}\label{sec:universality}

Although the restriction on the exponents and especially the positivity of the
coefficients seem like strong conditions, such systems can nonetheless be quite
complex. In this section, we investigate the breadth of such equations.

\begin{prop}\label{prop:universal}
For any system of $\ell$ real polynomials in $n$ variables, there exists a
system of $\ell+1$ equations in $n+1$ variables, in the form~(\ref{eqn:main}),
such that the positive solutions $(x_1, \ldots, x_n)$ to the former system are
in bijection with the positive solutions $(x_1', \ldots, x_{n+1}')$ of the
latter, with $x_i' = x_i / x_{n+1}$.
\end{prop}

\begin{proof}
We write our system of equations as $\sum_{\alpha \in S} a_{i\alpha} x^\alpha =
0$ for $1 \leq i \leq \ell$, where $S \subset \NN^n$ is an arbitrary finite set
of exponents and $a_{i\alpha}$ are any real numbers.  We let $d$ be the maximum
degree of any monomial $x^\alpha$ for $\alpha \in S$. We homogenize the
equations with a new variable~$x_{n+1}$. Explicitly, define $S' \subset
\NN^{n+1}$ to consist of $\alpha' = (\alpha, d - \sum_i \alpha_i)$ for all
$\alpha$ in~$S$ and we write $a_{i\alpha'} = a_{i \alpha}$. We add a new
equation with coefficients
$a_{\ell+1,\alpha} = 1$ for all $\alpha\in S'$ and $b_{\ell+1} = 1$. For this
system, we can clearly take $g_{1i} = 1$ and $d_1 = d$ to satisfy the condition
on exponents. Furthermore, for any positive solution $(x_1, \ldots, x_n)$ to the
original system of equations, $(x_1', \ldots, x_{n+1}')$ with $x_i' = x_i /
\big(\sum_\alpha x^\alpha\big)^{1/d}$ and $x_{n+1}' = 1/ \big(\sum_\alpha
x^\alpha \big)^{1/d}$ is a solution to
the homogenized system of equations.

Next, we add a multiple of the last equation to each of the others in order to
make all the coefficients positive. For each $1 \leq i \leq \ell$, choose a
positive $b_i > -\min_{\alpha}\{a_{i\alpha} \mid \alpha \in S'\}$, and define
$a_{i\alpha}' = a_{i\alpha} + b_i$. By construction, the resulting system has
all positive coefficients, and since the equations are formed from the previous
equations by elementary linear transformations, the set of solutions are the
same.
\end{proof}

The practical use of the construction in the proof of
Proposition~\ref{prop:universal} is mixed.
The first step, of homogenizing to deal with arbitrary sets of exponents, is a
straightforward way of guaranteeing the existence of the matrix~$g$. However,
for large systems, the second step tends to produce an ill-conditioned
coefficient matrix. In these cases, our algorithm converges very slowly.
Nonetheless, Proposition~\ref{prop:universal} shows that in the worst case,
systems satisfying our hypotheses can be as complicated as arbitrary polynomial
systems.

\begin{prop}
There exist bilinear equations in $2m$ variables with ${2m -2 \choose m-1}$
positive real solutions.
\end{prop}

\begin{proof}
We use a variation on the technique used to prove
Proposition~\ref{prop:universal}.

First, we pick $2m-2$ generic homogeneous linear functions $b_{1}, \ldots,
b_{2m-2}$ on $m$ variables.  By generic, we mean for any $m$ of the $b_k$, the
only simultaneous solution of all $m$ linear equations is the trivial one.  This
genericity implies that any $m-1$ of the $b_k$ define a point in $\PP^{m-1}$
By taking a linear changes of coordinates in each set of variables,
we can assume that all of these points are positive, i.e.\ have a representative
consisting of all positive real numbers.

Then we consider the system of equations
\begin{align}
 b_k(x_1, \ldots, x_m) \cdot b_k(x_{m+1}, \ldots, x_{n})
&= 0, \mbox{ for }1 \leq k \leq 2m-2  \label{eqn:bilinear-hom} \\
(x_1 + \ldots + x_m) (x_{m+1} + \ldots + x_{2m}) &= 1
\label{eqn:bilinear-dehom} \\
x_1 + \ldots + x_m &= 1. \label{eqn:linear-dehom}
\end{align}
The equations~(\ref{eqn:bilinear-hom}) are bihomogeneous and so we can think
of their
solutions in $\mathbb P^{m-1} \times \mathbb P^{m-1}$. There are exactly
${2m-2 \choose m-1}$ positive real solutions, corresponding to the subsets $A \subset
[2m-2]$ of size $m-1$. For any such~$A$, there is a unique, distinct solution
satisfying $b_k(x_1, \ldots, x_m) = 0$ for all $k$ in
$A$ and $b_k(x_{m+1}, \ldots, x_{2m}) = 0$ for all $k$ not in $A$.
By assumption, for each solution, all the coordinates can be chosen to be
positive.
The last two equations~(\ref{eqn:bilinear-dehom}) and~(\ref{eqn:linear-dehom})
dehomogenize the system in a way such that there are ${2m-2 \choose m-1}$
positive real solutions. Finally, as in the last paragraph of the proof of
Proposition~\ref{prop:universal}, we can add multiples
of~(\ref{eqn:bilinear-dehom}) to the equations~(\ref{eqn:bilinear-hom}) in order
to make all the coefficients positive.
\end{proof}

\section*{Acknowledgments}
We thank Bernd Sturmfels for reading a draft of this manuscript. This work was
supported by the Defense Advanced Research Projects Agency project ``Microstates
to Macrodynamics: A New Mathematics of Biology'' and the U.S.~National Science
Foundation (DMS-0456960).

\bibliographystyle{plain}
\bibliography{bilinear}

\end{document}